\documentclass[12pt]{amsart}
\usepackage[latin1]{inputenc}
\usepackage{color}
\usepackage{enumerate}
\usepackage{amssymb}
\usepackage[all]{xy}
\usepackage{hyperref}

\newtheorem{thm}{Theorem}[section]
\newtheorem{cor}[thm]{Corollary}
\newtheorem{lem}[thm]{Lemma}

\newtheorem{prop}[thm]{Proposition}
\theoremstyle{definition}

\newtheorem{question}[thm]{Question}
\theoremstyle{remark}
\newtheorem{rem}[thm]{Remark}

\newtheorem{example}{Example}
\numberwithin{equation}{section}

\newcommand{\eps}{\varepsilon}

\newcommand{\R}{\mathbb{R}}
\newcommand{\N}{\mathbb{N}}

\DeclareMathOperator{\spann}{span}

\DeclareMathOperator{\diam}{diam}

\DeclareMathOperator{\Ext}{Ext}

\newcommand{\pten}{\ensuremath{\widehat{\otimes}_\pi}}

\usepackage{color}
\usepackage{enumerate}
\begin{document}
\setcounter{tocdepth}{1}


\title{Octahedral norms in free Banach lattices}

\author[Dantas]{Sheldon Dantas}
\address[Dantas]{Department of Mathematics, Faculty of Electrical Engineering, Czech Technical University in Prague, Technick\'a 2, 166 27, Prague 6, Czech Republic \newline
	\href{http://orcid.org/0000-0001-8117-3760}{ORCID: \texttt{0000-0001-8117-3760} } }
\email{\texttt{gildashe@fel.cvut.cz}}

\author[Mart\'inez-Cervantes]{Gonzalo Mart\'inez-Cervantes}
\address[Mart\'inez-Cervantes]{Universidad de Murcia, Departamento de Matem\'{a}ticas, Campus de Espinardo 30100 Murcia, Spain
\newline
\href{http://orcid.org/0000-0002-5927-5215}{ORCID: \texttt{0000-0002-5927-5215} } }	
	
\email{gonzalo.martinez2@um.es}

\author[Rodr\'iguez Abell\'an]{Jos\'e David Rodr\'iguez Abell\'an}
\address[Rodr\'iguez Abell\'an]{Universidad de Murcia, Departamento de Matem\'{a}ticas, Campus de Espinardo 30100 Murcia, Spain 	\newline
	\href{https://orcid.org/0000-0002-2764-0070}{ORCID: \texttt{0000-0002-2764-0070} }}

\email{josedavid.rodriguez@um.es}

\author[Rueda Zoca]{Abraham Rueda Zoca}
\address[Rueda Zoca]{Departamento de An\'alisis Matem\'atico, 18071, Granada, Spain
	\newline
	\href{https://orcid.org/0000-0003-0718-1353}{ORCID: \texttt{0000-0003-0718-1353} }}
\email{\texttt{abrahamrueda@ugr.es}}
\urladdr{\url{https://arzenglish.wordpress.com}}

\thanks{S. Dantas was supported by the project OPVVV CAAS CZ.02.1.01/0.0/0.0/16\_019/0000778 and by the Estonian Research Council grant PRG877. G. Mart\'inez-Cervantes and J. D. Rodr\'iguez Abell\'an were supported by the project MTM2017-86182-P (Government of Spain, AEI/FEDER, EU) and the project 20797/PI/18 by Fundaci\'{o}n S\'{e}neca, ACyT Regi\'{o}n de Murcia. The research of G. Mart\'inez-Cervantes has been co-financed by the European Social Fund (ESF) and the Youth European Initiative (YEI) under the Spanish Seneca Foundation (CARM).
J. D. Rodr\'iguez Abell\'an was supported by FPI contract of Fundaci\'on S\'eneca, ACyT Regi\'{o}n de Murcia.
The research of A. Rueda Zoca was supported by MICINN (Spain) Grant PGC2018-093794-B-I00 (MCIU, AEI, FEDER, UE), by Junta de Andaluc\'ia Grant A-FQM-484-UGR18 and by Junta de Andaluc\'ia Grant FQM-0185}



\begin{abstract}
In this paper, we study octahedral norms in free Banach lattices $FBL[E]$ generated by a Banach space $E$. We prove that if $E$ is an $L_1(\mu)$-space, a predual of von Neumann algebra, a predual of a JBW$^*$-triple, the dual of an $M$-embedded Banach space, the disc algebra or the projective tensor product under some hypothesis, then the norm of $FBL[E]$ is octahedral. We get the analogous result when the topological dual $E^*$ of $E$ is almost square. We finish the paper by proving that the norm of the free Banach lattice generated by a Banach space of dimension $ \geq 2$ is nowhere Fr\'echet differentiable. Moreover, we discuss some open problems on this topic.
\end{abstract}

\maketitle

\section{Introduction}

Recently, a new concept was introduced by A. Avil\'es, J. Rodr\'iguez, and P. Tradacete in order to establish strong connections between Banach lattices and Banach spaces \cite{ART18}. If $E$ is a Banach space, then the free Banach lattice generated by $E$ is a Banach lattice that contains a subspace which is linearly isometric to $E$ and such that the elements of this new space can be seen as lattice-free generators (see Section \ref{sect:notation} for a precise definition). Since then, it has been intensively studied for the purpose of better understanding how the vector space, the order, and the topological structures are related. In particular, in \cite{ART18} they answered a question of J. Diestel, providing an example of a Banach lattice which is weakly compactly generated as a lattice but not as a Banach space. Namely, they proved that the free Banach lattice generated by $\ell_p(\Gamma)$ contains an isomorphic copy of $\ell_1(\Gamma)$ whenever $1 < p \leq 2$ and, therefore, it is not weakly compactly generated whenever $\Gamma$ is uncountable \cite[Theorem 5.4]{ART18}. Notice that it is weakly compactly generated as a lattice since the canonical copy of $\ell_p(\Gamma)$ inside it generates the whole space. Surprisingly, later on it was proved that the free Banach lattice generated by $\ell_p(\Gamma)$ is weakly compactly generated whenever $2<p<\infty$  \cite{ATI19}.

In general, it is not clear what properties a free Banach lattice can inherit from its generator Banach space. In fact, it is a difficult task to deal with problems involving geometric aspects of the free Banach lattice generated by a Banach space. This is so because a good description of its elements is not yet known; although we do have an explicit formula for its norm, it requires somehow an extra effort when it comes to geometric properties and, in particular, when it comes to the study of octahedrality of the norm. Recall that the norm of a Banach space $E$ is said to be octahedral if, for every finite-dimensional subspace $Z$ of $E$ and every $\eps > 0$, there exists $x \in S_E$ such that $\|z + \lambda x\| \geq (1 - \eps)(\|z\| + |\lambda|)$ for every $z \in Z$ and $\lambda \in \R$. That is, every finite-dimensional subspace $Z$ of $E$ has a $1$-dimensional complement $V$ in $E$ in such a way that, when one sees $Z\oplus V$ in $E$, it is roughly speaking an $\ell_1$-sum up to $\varepsilon$. Indeed, $\ell_1$ has octahedral norm. Also, it is well-known that if $E$ has the Daugavet property, then both $E$ and $E^*$ have octahedral norm.

From our point of view, the difficulty of dealing with the norm of a free Banach lattice should be compared to the projective norm on the projective tensor product between Banach spaces. According to P.~Tradacete (private communication), analogously to the projective norm, free Banach lattices do not have a good behavior when it comes to isometric subspaces in general. It seems to be reasonable, therefore, to try transferring the techniques from tensor products into this new context by looking at the examples where octahedrality of the norm in projective tensor products has a good behavior (see \cite{llr1, llr2, RZ}). For instance, it has been recently proven in \cite[Theorem 4.3]{llr2} that if $E$ is a non-reflexive $L$-embedded space and $F$ is a Banach space such that either $E^{**}$ or $F$ has the metric approximation property, then $E \pten F$ has an octahedral norm (this result should be compared to our Theorem \ref{theo:cuninxInfiinfi}); also, \cite[Corollary 2.9]{llr1} says that, under some conditions on a Banach space $E$, the projective norm on $E^* \pten F^*$ is octahedral when $E$ is almost square (this should be compared to our Theorem \ref{theo:dualASQ}); finally, by \cite[Proposition 4.1]{llr2}, if $F$ is a finite-dimensional Banach space and $E \pten F$ has an octahedral norm, then so does $E$ (this should be compared, which turns out to be an open question, to Question \ref{finite}).

Let us briefly describe the content of this article. In Section \ref{sect:notation}, we introduce the necessary notation and background. We also prove some results which relate the structures of both $FBL[E]$ and $FBL[E^{**}]$; these have direct implications in the following section. In Section \ref{sect:centrainfi}, we prove Theorem \ref{theo:cuninxInfiinfi}, which turns out to be one of the main results of the paper. This provides a large class of Banach spaces $E$ for which the norm of $FBL[E]$ is octahedral such as $L_1(\mu)$-spaces, preduals of von Neumann algebras, preduals of JBW$^*$-triples, duals of an $M$-embedded Banach spaces or some projective tensor product spaces (see Example \ref{exam:aplicunin}). In Section \ref{sect:ASQ}, we prove the other main result of the paper, Theorem \ref{theo:dualASQ}, where we obtain that the norm of $FBL[E]$ is octahedral whenever $E$ is a Banach space with almost square dual. Finally, we finish the paper with Section \ref{sect:opeque} by presenting some open questions related to the topic and some remarks. Let us point out that we prove in Proposition \ref{prop:bigccslice} that if $E$ is a finite-dimensional Banach space with dimension $\geq2$, then all the convex combinations of $w^*$-slices of the unit ball of $FBL[E]^*$ has diameter bigger than a fixed constant (depending only on $E$), which suggests that the norm of $FBL[E]$ can be octahedral even when $E$ is finite-dimensional with $\dim(E)\geq 2$ (see Question \ref{finite}). As a consequence, the norm of the free Banach lattice generated by a Banach space of dimension $ \geq 2$ is nowhere Fr\'echet differentiable.

\section{Notation and Preliminary Results} \label{sect:notation} 

In this paper, we consider only {\bf real} Banach spaces. Given a Banach space $E$, $B_E$ (respectively, $S_E$) stands for the closed unit ball (respectively, the unit sphere) of $E$. We denote by $E^*$ the topological dual space of $E$. By a {\it slice} of $B_E$, we mean a subset of the form
\begin{equation*} 
S(B_E,f,\alpha):=\{x\in B_E: f(x)>1-\alpha\},
\end{equation*} 
where $f\in S_{E^*}$ and $\alpha>0$. If $E$ is a dual space, say $E=F^*$, by a {\it $w^*$-slice} of $B_{E^*}$, we mean a slice of the form $S(B_E,y,\alpha)$ with $y\in F$. As we already have mentioned in the introduction, the norm of a Banach space $E$ is said to be  {\it octahedral} if for every $\eps>0$ and every finite-dimensional subspace $Z$ of $E$, there is $x \in S_E$ such that 
\begin{equation*} 
\|z+\lambda x\| \geq (1-\varepsilon)(\|z\| +|\lambda|),
\end{equation*}  
for every $z\in Z$ and $\lambda \in \R$ (see \cite{dgz}). Let us notice also that a Banach space $E$ has an octahedral norm if and only if given $x_1,\ldots, x_n\in S_E$ and $\varepsilon>0$, we can find $x\in S_E$ such that $\Vert x_i+x\Vert>2-\varepsilon$, for every $i\in\{1,\ldots, n\}$ (see \cite[Proposition 2.1]{hlp}). We will make use of this fact throughout the whole paper without any explicit mention.

We say that $E$ satisfies the \textit{strong diameter two property} (SD2P, for short), if every convex combination of slices of the closed unit ball of $E$ has diameter two. If $E$ is itself a dual Banach space, then the \textit{$w^*$-strong diameter two property} ($w^*$-SD2P, for short) can be defined as usual just invoking convex combinations of $w^*$-slices of the unit ball of $E$. Examples of Banach spaces which satisfy the strong diameter two properties are infinite-dimensional uniform algebras \cite{nywe}, Banach spaces satisfying Daugavet property \cite{sh} or non-reflexive $M$-embedded spaces \cite{gines}.

It is known that the norm on a Banach space $E$ is octahedral if and only if $E^*$ satisfies the $w^*$-SD2P (see \cite[Theorem 2.1]{blrjfa}). It is also known that a Banach space $E$ has an equivalent octahedral norm if and only if $E$ contains an isomorphic copy of $\ell_1$ (see \cite{G}). Very recently, the previous result was improved in \cite{ll}, where it was proved that if a Banach space $E$ is separable and contains an isomorphic copy of $\ell_1$, then there exists an equivalent renorming of $E$ such that the bidual norm is octahedral.

\bigskip

Now, let us describe formally the objects we are working with in the paper. If $X, Y$ are Banach lattices, we say that an operator $T \colon X \longrightarrow Y$ is a {\it Banach lattice homomorphism}, or simply, a {\it lattice homomorphism}, if it preserves the lattice operations, that is, $T(x \wedge y) = T(x) \wedge T(y)$ and $T(x \vee y) = T(x) \vee T(y)$ for every $x, y \in X$. If $E$ is a Banach space, then the {\it free Banach lattice generated by $E$} is a Banach lattice $X$ together with a bounded operator $\phi_E \colon E \longrightarrow X$ with the property that for every Banach lattice $Y$ and every bounded operator $T\colon E \longrightarrow Y$, there is a unique Banach lattice homomorphism $\hat{T}\colon X \longrightarrow Y$ such that $T = \hat{T} \circ \phi_E$ and $\|\hat{T}\| = \|T\|$. Note that this definition generalizes the notion of a free Banach lattice generated by a set with no extra structure (see \cite[Corollary 2.8]{ART18}). Let $x \in E$ be given and denote by $\delta_x\colon E^* \longrightarrow \R$ the {\it evaluation function} given by $\delta_x(x^*) = x^*(x)$ for every $x^* \in E^*$. For a function $f\colon E^* \longrightarrow \R$, we define the norm
\begin{equation*} \|f\|_{FBL[E]} := \sup \left\{ \sum_{i=1}^n |f(x_i^*)|: n \in \N, x_1^*, \ldots, x_n^* \in E^*, \sup_{x \in B_E} \sum_{i=1}^n |x_i^*(x)| \leq 1 \right\}.
\end{equation*}

\bigskip
The following remark will play an important role throughout this paper and will be used without any explicit reference.
\begin{rem}
\label{RemSupBall}
Given $x_1^*,\ldots, x_n^*\in E^*$, the condition $\sup \sum_{i=1}^n \vert x_i^*(x)\vert\leq 1$ over all elements $x$ in $B_E$ is equivalent to the condition that $\Vert \sum_{i=1}^n \xi_i x_i^*\Vert\leq 1$ holds for every choice of signs $\xi_1,\ldots, \xi_n\in \{-1,1\}$.
\end{rem} 

 By \cite[Theorem 2.4]{ART18}, the free Banach lattice generated by the Banach space $E$ is the closure of the vector lattice generated by the set $\{\delta_x: x \in E\}$ in $\R^{E^*}$ under the above norm. In other words, $FBL[E]$ is the Banach lattice generated by $\{\delta_x: x \in E\}$ in the Banach lattice $\R^{E^*}$ with the norm $\|\cdot\|_{FBL[E]}$, the pointwise order, and the pointwise operations. Therefore, given a Banach space $E$, it is always possible to construct the associated free Banach lattice. We would like to highlight that the natural identification of $E$ in $FBL[E]$ is given by the map $\phi_E: E \longrightarrow FBL[E]$ defined by $x \mapsto \delta_x$ for every $x \in E$ (it is a linear isometry between $E$ and its image in $FBL[E]$). Moreover, all the functions in $FBL[E]$ are positively homogeneous (i.e., $f(\lambda x^*)=\lambda f(x^*)$ for every $\lambda \geq0$ and every $x^* \in E^*$) and $w^*$-continuous when restricted to $B_{E^*}$. 
 A function $f\colon E^* \longrightarrow \R$ is said to \textit{depend on finitely many coordinates} $x_1,\ldots,x_n \in E$ if $f(x^*)=f(y^*)$ whenever $x^*(x_i)=y^*(x_i)$ for every $i\leq n$. Notice that each $\delta_x$ depends only on one coordinate (just $x$). Since every function in $FBL[E]$ can be approximated by a finite lattice linear combination of elements of the form $\delta_{x_i}$, the following holds true.
 \begin{rem}
 \label{RemDensidadFuncionesQueDependenDeUnaCantidadFinitaDeCoordenadas}
 Every function in $FBL[E]$ can be approximated by a function depending on finitely many coordinates.
 \end{rem}

For every functional $x^* \colon E \longrightarrow \R$, we can define the functional $\delta_{x^*} \colon FBL[E] \longrightarrow \R$ given by the formula $\delta_{x^*}(f)=f(x^*)$ for every $f\in FBL[E]$. Since the functions of $FBL[E]$ are $w^*$-continuous, it follows that the map $\iota \colon E^* \longrightarrow FBL[E]^*$ defined by $x^* \mapsto \delta_{x^*}$ for every $x^* \in E^*$ is $w^*$-to-$w^*$ continuous, although it is not linear. 
Throughout this paper we use the following lower and upper bounds for a linear combination of elements of this form.
\begin{lem}
\label{LemNormInequalities}
Let $x_1^*,\ldots,x_n^* \in E^*$ and $a_1,\ldots,a_n \in \R$. Then,
\begin{eqnarray*} 
\| a_1x_1^* + \ldots + a_nx_n^* \| &\leq& \| a_1\delta_{x_1^*}+\ldots+ a_n\delta_{x_n^*} \|_{FBL[E]^*} \\
&\leq& \max\{\| \mu_1 a_1x_1^*+\ldots+ \mu_n a_nx_n^* \|:\mu_1,\ldots,\mu_n \in \{-1,1\} \}.
\end{eqnarray*} 
\end{lem}
\begin{proof}
The first inequality follows from 
\begin{eqnarray*} 	
\| a_1\delta_{x_1^*}+\ldots+ a_n\delta_{x_n^*} \|_{FBL[E]} &\geq& \sup_{x\in B_E} \left(a_1\delta_{x_1^*}+\ldots+ a_n\delta_{x_n^*}\right)(\delta_x) \\
&=& \sup_{x\in B_E}  \left(a_1x_1^*(x)+\ldots+ a_nx_n^*(x)\right) \\
&=& \| a_1x_1^* + \ldots + a_nx_n^* \|.
\end{eqnarray*}
To prove the second inequality, set 
\begin{equation*} 
\alpha:=\max\{\| \mu_1 a_1x_1^*+\ldots+ \mu_n a_nx_n^* \|:\mu_i \in \{-1,1\} \}.
\end{equation*} 
Then, by Remark \ref{RemSupBall},
$$  \sup_{x \in B_E} \sum_{i=1}^n \frac{|a_ix_i^*(x)|}{\alpha}  =1 .$$
Thus, for any $f\in FBL[E]$, by the definition of the norm in $FBL[E]$, we have 
\begin{eqnarray*} 
\|f\|_{FBL[E]} \geq \sum_{i=1}^n a_if\left(\frac{x_i^*}{\alpha} \right) &=& \frac{1}{\alpha} \sum_{i=1}^n a_if(x_i^*) \\
&=&\frac{1}{\alpha} (a_1\delta_{x_1^*}+\ldots+ a_n\delta_{x_n^*})(f).
\end{eqnarray*} 
So, $\| a_1\delta_{x_1^*}+\ldots+ a_n\delta_{x_n^*} \|_{FBL[E]^*} \leq \alpha$ and this proves the second inequality.
\end{proof}

\bigskip
Recall that, for a Banach space $E$, we say that a subset $A$ of $B_{E^*}$ is {\it norming} if, for every $x \in E$, we have $\|x\| = \sup \{|\varphi(x)|: \varphi \in A \}$. It follows from the definition of the norm $\|\cdot\|_{FBL[E]}$ that the set 
\begin{equation}\label{equa:norming}
A:=\left\{ \sum_{i=1}^n \xi_i \delta_{x_i^*}: x_i^*\in E^*, \xi_i\in \{-1,1\},  \sup\limits_{x\in B_E}  \sum_{i=1}^n \vert x_i^*(x)\vert\leq 1 \right\} \subseteq B_{FBL[E]^*}
\end{equation}
is norming for $FBL[E]$. By a separation argument, we have that $\overline{co}^{w^*}(A)=B_{FBL[E]^*}$, where $\overline{co}^{w^*}(A)$ denotes the $w^*$-closed convex hull of $A$. Indeed, using the fact that, for every $x^*\in E^*$ and every nonnull $\lambda\in \R$, we have that $\lambda \delta_{x^*} = \xi \delta_{|\lambda|x^*}$, where $\xi$ is the sign of $\lambda$, it follows that $A$ is a convex set. Thus, indeed, $$ \overline{A}^{w^*}=\overline{co}^{w^*}(A)=B_{FBL[E]^*}.$$

Next we prove that the set $A$, although algebraically written in the same way as above but now viewed as a subset of $FBL[E^{**}]^*$, is norming for this space.

\begin{lem}\label{lemma:normanbidu}
Let $E$ be a Banach space. Consider the set 
\begin{equation*} 
A:= \left\{\sum_{i=1}^n \xi_i\delta_{x_i^*}: x_i^*\in E^*, \xi_i\in \{-1,1\}, \sup\limits_{x\in B_E} \sum_{i=1}^n \vert x_i^*(x)\vert\leq 1 \right\}.
\end{equation*} 
Then, we have that
\begin{itemize}
	\item[(1)] $A$ is a subset of $FBL[E^{**}]^*$.
	\item[(2)] $A$ is norming for $FBL[E^{**}]$.
	\item[(3)] $FBL[E]$ is isometric to a sublattice of $FBL[E^{**}]$.  
\end{itemize}
\end{lem}

\begin{proof} Item (1) is immediate from the definition. Let us prove (2). Pick $f\in FBL[E^{**}]$ depending on finitely many coordinates, say $\{x_1^{**},\ldots, x_k^{**}\}$. Let $\varepsilon>0$ be given. By Remark \ref{RemDensidadFuncionesQueDependenDeUnaCantidadFinitaDeCoordenadas}, it is enough to find $a\in A$ with $a(f)>1-\varepsilon$. To this end, take an element $\sum_{i=1}^n \xi_i \delta_{x_i^{***}}$ in the norming set 
\begin{equation*} 
\left\{\sum_{i=1}^n \xi_i\delta_{x_i^{***}}: x_i^{***}\in E^{***}, \xi_i\in \{-1,1\}, \sup\limits_{x^{**}\in B_{E^{**}}} \sum_{i=1}^n \vert x_i^{***}(x^{**})\vert\leq 1 \right\} 
\end{equation*} 
such that 
\begin{equation*} 
\sum_{i=1}^n \xi_i f(x_i^{***})>\Vert f\Vert-\varepsilon. 
\end{equation*} 
Notice that 
$$\left\Vert \sum_{i=1}^n \sigma_i x_i^{***}\right\Vert\leq 1$$
holds for every choice of signs $\sigma_1,\ldots, \sigma_n\in \{-1,1\}$. Define $H:= \spann \{x_1^{***},\ldots, x_n^{***}\}$ and $F:=\spann\{x_1^{**},\ldots, x_k^{**}\}$, and pick any $\eta>0$. By the Principle of Local Reflexivity (see, for instance, \cite[Theorem 9.15]{fab}), we can find an operator $T\colon H \longrightarrow E^*$ such that 
\begin{enumerate}
\item $(1-\eta)\Vert e\Vert\leq \Vert T(e)\Vert\leq (1+\eta)\Vert e\Vert$ holds for every $e\in H$ and
\item $x^{**}(T(x^{***}))=x^{***}(x^{**})$ holds for every $x^{***}\in H$ and every $x^{**}\in F$.
\end{enumerate}
Define $x_i^*:=T(x_i^{***})$. On the one hand, notice that given any choice of signs $\sigma_1,\ldots, \sigma_n\in \{-1,1\}$, we get that
$$\left\Vert \sum_{i=1}^n \sigma_i x_i^*\right\Vert=\left\Vert T\left(\sum_{i=1}^n \sigma_i x_i^{***}\right)\right\Vert\leq (1+\eta)\left\Vert \sum_{i=1}^n \sigma_i x_i^{***}\right\Vert\leq 1+\eta.$$
This implies that 
\begin{equation*} 
a:=\sum_{i=1}^n \xi_i \delta_{\frac{x_i^*}{1+\eta}}\in A. 
\end{equation*} 
Moreover, notice that $x_i^*=x_i^{***}$ on $F$. Since $f$ depends on the coordinates $x_1^{**},\ldots, x_k^{**}$, we get that $f(x_i^*)=f(x_i^{***})$ holds for every $i\in\{1,\ldots, n\}$. Finally, since $f$ is positively homogeneous, we get that
$$a(f)=\frac{\sum_{i=1}^n \xi_i f(x_i^{***})}{1+\eta}>\frac{1-\varepsilon}{1+\eta}.$$

Let us now prove item (3). Let $i$ be the canonical isometric embedding of $E$ into $E^{**}$. Consider $\phi_{E^{**}}\colon E^{**} \longrightarrow FBL[E^{**}]$ the canonical isometry. By the definition of $FBL[E]$, there exists a unique lattice homomorphism $T \colon  FBL[E] \longrightarrow FBL[E^{**}]$ such that $T \circ \phi_E = \phi_{E^{**}}\circ i$ and $\|T\| = \|\phi_{E^{**}}\circ i\| = 1$. 

\vspace{0.2cm}
\noindent
{\bf Claim}: $T$ is an isometry. 
\vspace{0.2cm}

Indeed, first notice that since $T \circ \phi_E = \phi_{E^{**}}\circ i$, we have that $T(\delta_x) = \delta_x$ for every $x \in E$. Let $a := \sum_{i=1}^n \xi_i \delta_{x_i^*} \in A$ be given. Notice that if $f \in \{ \delta_x: x \in E\}$, then 
	\begin{equation*}
	a(T(f)) = a(f),
	\end{equation*} 
	which implies that $a(T(f)) = a(f)$ for every $f$ which is in the vector lattice generated by $\{\delta_x: x\in E\}$. Now, we use the fact that $A$ is norming for both $FBL[E]$ and $FBL[E^{**}]$ to see that
	\begin{equation*}
	\|T(f)\|_{FBL[E^{**}]} = \sup_{a \in A} a(T(f)) = \sup_{a \in A} a(f) = \|f\|_{FBL[E]} 
	\end{equation*} 
	holds for every $f$ in the vector lattice generated by $\{\delta_x: x\in E\}$. This implies that $T$ is an isometry in a dense subspace of $FBL[E]$. Then, $T$ is an isometry in the whole space.
\end{proof}

By item (3) of Lemma \ref{lemma:normanbidu} we can identify, in a canonical way, $FBL[E]$ as an isometric sublattice of $FBL[E^{**}]$, for every Banach space $E$. We will make use of this fact without any further explicit reference throughout the paper.

\bigskip

The following lemma will be useful to inherit an octahedrality condition on a free Banach lattice $FBL[E]$ from $FBL[E^{**}]$. This result should be compared to the fact that the norm of $X$ is octahedral if the norm of $X^{**}$ is octahedral (it follows, for instance, from the Principle of Local Reflexivity). This will play an important role in the proof of Theorem \ref{theo:cuninxInfiinfi} in the next section.

\begin{lem} \label{lemma:octabaja} Let $E$ be a Banach space, $f_1,\ldots, f_k\in S_{FBL[E]}$, and $\eps>0$. Assume that there exists $x^{**} \in B_{E^{**}}$ such that
\begin{equation*} 
\Vert f_i+\delta_{x^{**}}\Vert_{FBL[E^{**}]}>2-\varepsilon.
\end{equation*} 	
Then, there exists $x\in B_E$ such that
\begin{equation*} 
\Vert f_i+\delta_x\Vert_{FBL[E]}>2- \varepsilon,
\end{equation*} 
holds for every $i\in\{1,...,k\}$.
\end{lem}

\begin{proof} Let $i \in \{1, \ldots, k\}$ be fixed and assume that there exists $x^{**} \in B_{E^{**}}$ such that $\|f_i + \delta_{x^{**}}\|_{FBL[E^{**}]} > 2 - \eps$. Then, by Lemma \ref{lemma:normanbidu}, there exists $\sum_{j=1}^{n_i} \xi_{ij} \delta_{x_{ij}^*} \in A$ such that
\begin{equation*}
\sum_{j=1}^{n_i} \xi_{ij} (f_i(x_{ij}^*) + \delta_{x^{**}}(x_{ij}^*))=\sum_{j=1}^{n_i} \xi_{ij} (f_i(x_{ij}^*) + x^{**}(x_{ij}^*)) > 2 - \eps.
\end{equation*}
By using the $w^*$-denseness of $B_E$ in $B_{E^{**}}$, we may find a net $(x_s) \stackrel{w^*}{\longrightarrow} x^{**}$. By the previous convergence, we get that
$$\sum_{j=1}^{n_i} \xi_{ij} (f_i(x_{ij}^*) + x_s(x_{ij}^*))\longrightarrow \sum_{j=1}^{n_i} \xi_{ij} (f_i(x_{ij}^*) + x^{**}(x_{ij}^*))>2-\eps.$$
Hence, we can find $x=x_s$, for $s$ large enough, so that 
\begin{equation*} 
\sum_{j=1}^{n_i} \xi_{ij} (f_i(x_{ij}^*) + x(x_{ij}^*))>2-\eps. 
\end{equation*} 
Consequently,
\begin{eqnarray*}
\|f_i + \delta_x\|_{FBL[E]} &\geq& \sum_{j=1}^{n_i} \xi_{ij} (f_i(x_{ij}^*) + \delta_x(x_{ij}^*)) \\
&=& \sum_{j=1}^{n_i} \xi_{ij} (f_i(x_{ij}^*) + x(x_{ij}^*)) > 2 -  \eps,
\end{eqnarray*}
as desired.
\end{proof}

\section{Octahedral norms in free Banach lattices in terms \\ of the Cunningham algebra} \label{sect:centrainfi}

Let us start this section by enunciating the main result followed by some examples that show which Banach spaces satisfy the conditions of the theorem. Afterwards, we give the necessary definitions and background in order to prove Theorem \ref{theo:cuninxInfiinfi}. We follow the notation from \cite{BR}; given a Banach space $E$, we denote by $C(E)$ the Cunningham algebra of $E$ (see the paragraph preceding  Theorem \ref{theo:cunialgespa} for a formal definition) and by $E^{(\infty}$ the completion of the normed space $\bigcup_{n=0}^{\infty} E^{(2n}$, where $(E^{(2n})_{n=0}^{\infty}$ is the sequence of even duals such that $E \subseteq E^{**} \subseteq E^{(4} \subset \ldots \subseteq E^{(2n} \subseteq \ldots$.

\begin{thm} \label{theo:cuninxInfiinfi} Let $E$ be a Banach space and suppose that $C(E^{(\infty})$ is infinite-dimensional. Then, given $f_1,\ldots, f_n\in S_{FBL[E]}$ and $\varepsilon>0$, there exists an element $x\in B_E$ such that
\begin{equation*} 
\|f_i+\delta_x\|_{FBL[E]} > 2- \eps 
\end{equation*} 
holds for every $i\in\{1,\ldots, n\}$. In particular, the norm of $FBL[E]$ is octahedral. 	
\end{thm} 

Let us exhibit some examples where Theorem \ref{theo:cuninxInfiinfi} applies.

\begin{example}\label{exam:aplicunin} $C(E^{(\infty})$ is infinite-dimensional (and so the norm of $FBL[E]$ is octahedral) if the Banach space $E$ satisfies one of the following conditions.
\begin{enumerate}
\item $E$ is a non-reflexive $L$-embedded Banach space (see  \cite[Proposition 3.4]{AB}). Let us recall that a Banach space $E$ is said to be \textit{$L$-embedded} if $E^{**}=E\oplus_1 Z$ for some subspace $Z$ of $E^{**}$. Examples of $L$-embedded Banach spaces are $L_1(\mu)$-spaces, preduals of von Neumann algebras, preduals of real or complex $JBW^*$-triples, duals of $M$-embedded Banach spaces or the disk algebra (we refer the reader to the references \cite[Example IV.1.1]{hww} and \cite[Proposition 2.2]{blpr} for formal definitions and details).

\vspace{0.2cm} 

\item $E$ is the projective tensor product $E_1\widehat{\otimes}_\pi E_2$ (which is the topological dual of the set of all bounded operators $\mathcal{L}(E_1, E_2^*)$) whenever $C(E_1)$ or $C(E_2)$ is infinite-dimensional (it follows from \cite[item (4), p.~851]{AB} and from the fact that, given any Banach space $F$, we have that $C(F)$ and $Z(F^*)$ are linearly isometric (see \cite[Theorems 5.7 and 5.9]{B}). For the necessary background on projective tensor products, we suggest the reader to go through Chapter 2 of \cite{rya}. 
\end{enumerate}
\end{example}

In what follows, we introduce the necessary background in order to prove Theorem \ref{theo:cuninxInfiinfi}. Given a family $\{E_i\}_{i \in I}$ of Banach spaces, we denote by $\prod_{i \in I}^{\infty} E_i$ the Banach space of elements $x \in \prod_{i \in I} E_i$ such that $\sup_{i \in I} \|x(i)\| < \infty$ endowed with the sup-norm. By a {\it function module}, we mean (the third coordinate of) a triple $(K, (E_t)_{t \in K}, E)$, where $K$ is a non-empty compact Hausdorff topological space (called the {\it base space}), $(E_t)_{t \in K}$ is a family of Banach spaces, and $E$ is a closed $\mathcal{C}(K)$-submodule of the $\mathcal{C}(K)$-module $\prod_{t \in K}^{\infty} E_t$ such that: 
\begin{itemize}
	\item[(i)] the function $t \mapsto \|x(t)\|$ from $K$ to $\R$ is upper semicontinuous for every $x \in E$.
	\item[(ii)] $E_t = \{x(t): x \in E \}$ for every $t \in K$.
	\item[(iii)] the set $\{t \in K: E_t \not= 0\}$ is dense in $K$.
\end{itemize}
We notice that $\|x\| = \sup_{t \in K} \|x(t)\|$ for every $x \in E$. We refer the reader to the book \cite[p.~75]{B} for basic results on function modules. We denote by $\mathcal{L}(E)$ the space of all bounded linear operators on $E$. By a {\it multiplier} on $E$, we mean an element $T \in \mathcal{L}(E)$ such that every extreme point of $B_{E^*}$ becomes an eigenvector for $T^*$. Then, given a multiplier $T$ on $E$ and an extreme point $p$ of $B_{E^*}$, there exists a number $a_T(p)$ satisfying 
\begin{equation*}
p \circ T = T^*(p) = a_T(p) p.
\end{equation*}
The {\it centralizer} of $E$, denoted by $Z(E)$, is defined  as the set of all multipliers on $E$ (recall that we are considering only real Banach spaces). $Z(E)$ is a closed subalgebra of $\mathcal{L}(E)$ isometrically isomorphic to $\mathcal{C}(K_E)$ for some Hausdorff compact space $K_E$. Moreover, $E$ can be seen as a function module whose base space is precisely $K_E$ and such that the elements of $Z(E)$ are precisely the operators of multiplication by the elements of $\mathcal{C}(K_E)$ (see, for instance, \cite[Theorem 4.14]{B}).

We will be making use of the following fact.

\begin{lem} \cite[Lemma 2.1]{BR} \label{extreme} Let $E$ be a Banach space. Let $(K, (E_t)_{t \in K}, E)$ be a function module and let $p$ be an extreme point of $B_E$. Then, for every $t \in K$, we have $\|p(t)\| = 1$. 
\end{lem}

\bigskip

Next, we give a precise definition of the Cunningham algebra. For a Banach space $E$,  an {\it $L$-projection} on $E$ is a (linear) projection $P: E \longrightarrow E$ satisfying $ \Vert x
\Vert = \Vert  P(x) \Vert + \Vert x- P(x) \Vert$  for every $x \in E$.  In such a case, we will say that the subspace $P(E)$ is an {\it $L$-summand} of $E$. Let us notice that the composition of
two $L$-projections on $E$ is an $L$-projection (see \cite[Proposition~1.7]{B}). So, the closed linear subspace of $\mathcal{L}(E)$ generated by all $L$-projections on $E$ is a subalgebra of $
\mathcal{L}(E)$. This algebra, denoted by $C(E)$, is called the {\it Cunningham algebra} of $E$. It is known that $C(E)$ is linearly isometric  to $Z(E^*)$ (see \cite[Theorems 5.7 and 5.9]{B}).

\bigskip

Before giving the proof of Theorem \ref{theo:cuninxInfiinfi}, we need the following analogous result which assumes that the Cunningham algebra of the Banach space $E$ itself is infinite-dimensional.
Recall that if $(f_n)_{n\in \N}$ is a uniformly bounded sequence in   $\mathcal{C}(K)$ such that $(f_n)_{n\in \N}$ converges to zero pointwise, then $f_n$  converges to zero in the weak topology of $\mathcal{C}(K)$ (see, for instance, \cite[Theorem 12.1]{fab}). We will use this fact in the proof of Theorem \ref{theo:cunialgespa} below.

\begin{thm} \label{theo:cunialgespa} Let $E$ be a Banach space and suppose that $C(E)$ is infinite-dimensional. Then, for every $f_1,\ldots, f_k\in S_{FBL[E]}$ and every $\varepsilon>0$, there exists $x\in S_E$ such that
$$\|f_i+\delta_x\|_{FBL[E]} > 2-\varepsilon$$
holds for every $i\in\{1,\ldots, k\}$. In particular, the norm of $FBL[E]$ is octahedral. 
\end{thm}

\begin{proof} 	
	
Let $k\in\mathbb N$, $f_1,\ldots, f_k\in S_{FBL[E]}$, and $\varepsilon>0$. For every $i\in\{1,\ldots,k\}$, let $\sum_{j=1}^{n_i} \xi_{ij} \delta_{x_{ij}^*} \in A$ such that 
\begin{equation*} 
\left(\sum_{j=1}^{n_i} \xi_{ij} \delta_{x_{ij}^*}\right)(f_i)>1-\frac{\varepsilon}{2}.
\end{equation*} 
Since $Z(E^*)$ is infinite-dimensional (because $C(E)$ is infinite-dimensional), then $E^*$ can be seen as a $\mathcal{C}(K)$-function module whose base space $K$ is infinite and such that the elements of $Z(E^*)$ are precisely the operators of multiplication by the elements of $\mathcal{C}(K)$. Because of this, we may find a sequence $\{ \mathcal{O}_n\}_{n \in \N}$ of disjoint open sets in $K$. Given $n \in \N$, we take $t_n \in \mathcal{O}_n$. By the Urysohn lemma, we can find $h_n \in \mathcal{C}(K)$ such that
	\begin{equation*}
	0 \leq h_n \leq 1, \ \ \ h_n(t_n) = 1, \ \ \ \mbox{and} \ \ \ h_n|_{(L\setminus \mathcal{O}_n)} = 0.
	\end{equation*}
Let us notice that $\{h_n\}_{n \in \N}$ is a sequence of bounded functions which converges to zero pointwise. So, $h_n \stackrel{w}{\longrightarrow} 0$ in $\mathcal{C}(K)$. Consequently,
\begin{equation*}
(1 - h_n)x^* \stackrel{w}{\longrightarrow} x^*
\end{equation*}
for every $x^* \in E^*$. Now, let us note that the set of extreme points of $B_{E^*}$ is not empty by the Krein-Milman theorem, which allows us to take an extreme point $p \in \Ext(B_{E^*})$. For each $n \in \N$, let us define 
\begin{equation*}
y_{n, ij}^{*} := (1 - h_n) x_{ij}^* + h_n \sigma_{ij} \frac{p}{n_i},
\end{equation*} 
where $\eps_{ij} \sigma_{ij} = 1$ for every $i, j$.

\vspace{0.1cm}
\noindent
{\bf Claim}: For every $\eps_1, \ldots, \eps_n \in \{-1,1\}$, we have that
\begin{equation*}
\left\| \sum_{j=1}^{n_i} \eps_j y_{n,ij}^* \right\| \leq 1.
\end{equation*}

Indeed, for every $t \in K$, we have

\begin{eqnarray*}
\left\| \sum_{j=1}^{n_i} \eps_j y_{n,ij}^*(t)\right \| &=& \left\| \sum_{j=1}^{n_i} \eps_j \left( (1 - h_n(t)) x_{ij}^*(t) + h_n(t) \sigma_{ij} \frac{p(t)}{n_i} \right) \right\| \\
&\leq& |1 - h_n(t)| \left\| \sum_{j=1}^{n_i}   \eps_j x_{ij}^*(t) \right\| + |h_n(t)| \sum_{j=1}^{n_i}  \frac{\|p(t)\|}{n_i} \\
&\leq& (1 - h_n(t)) + h_n(t) \\ 
&=& 1.
\end{eqnarray*} 
Taking supremum in $t$, the claim is proved by the properties of a function module.

Now, notice that, fixed $i\leq k$ and $j \in \{1, \ldots, n_i\}$, we have that $\left(y_{n,ij}^*\right)_{n\in \N}$ converges to $x_{ij}^*$  in the $w$-topology of $E^*$, and therefore it also converges in the $w^*$-topology. This implies that
\begin{equation*}
\sum_{j=1}^{n_i} \xi_{ij} \delta_{y_{n, ij}^*}\stackrel{w^{*}}{\longrightarrow} \sum_{j=1}^{n_i} \xi_{ij} \delta_{x_{ij}^*}.
\end{equation*}
Therefore, we can fix $m\in\mathbb N$ such that 
\begin{equation} \label{eq1} 
\left(\sum_{j=1}^{n_i} \xi_{ij} \delta_{y_{m, ij}^*} \right)(f_i)>1-\frac{\varepsilon}{2}.
\end{equation}

\noindent
By Lemma \ref{extreme}, we have that $\Vert p(t_m)\Vert=1$. Thus, we can find $\varphi \in S_{E_{t_m}^*}$ such that $\varphi(p(t_m))>1-\frac{\varepsilon}{2}$. Define $\phi \colon E^*\longrightarrow \mathbb R$ by the equation
\begin{equation*} 
\phi(x^*)=\varphi (x^*(t_m)) \ \ (x^* \in E^*).
\end{equation*} 
It is clear that $\phi$ is linear, bounded, and that $\Vert \phi\Vert\leq 1$. In other words, $\phi\in B_{E^{**}}$. Hence, by Lemma \ref{lemma:normanbidu} and inequality (\ref{eq1}), we get that
\begin{equation*} 
\Vert f_i+\delta_\phi\Vert_{FBL[E^{**}]}\geq \left(\sum_{j=1}^{n_i} \xi_{ij} \delta_{y_{m, ij}^*} \right) (f_i+\delta_\phi)  > 1-\frac{\varepsilon}{2}+\sum_{j=1}^{n_i} \xi_{ij}\phi(y_{m,ij}^*),
\end{equation*}
where in the first inequality we are using that $\|\sum_{j=1}^{n_i} \xi_{ij} \delta_{y_{m, ij}^*}\| \leq 1$ by the previous claim and Lemma \ref{LemNormInequalities}.

On the other hand,
 
\begin{eqnarray*}
\sum_{j=1}^{n_i} \xi_{ij}\phi(y_{m,ij}^*) &=& \sum_{j=1}^{n_i} \xi_{ij}\varphi \left( (1-h_m)(t_m)x_{ij}^*(t_m)+h_m(t_m)\sigma_{ij} \frac{p(t_m)}{n_i} \right)\\
&=& \sum_{j=1}^{n_i}\xi_{ij} \sigma_{ij}\frac{\varphi(p(t_m))}{n_i} =\varphi(p(t_m)) > 1-\frac{\varepsilon}{2}.
\end{eqnarray*} 
Therefore, $\Vert f_i+\delta_\phi\Vert_{FBL[E^{**}]}>2-\varepsilon$. Finally, by Lemma \ref{lemma:octabaja}, there exists an element $x\in B_E$ such that $\Vert f_i+\delta_x\Vert_{FBL[E]}>2-\varepsilon$ holds for every $i\in\{1,\ldots, k\}$, as desired.
\end{proof}

Recall that $E^{(\infty}$ is the completion of the normed space $\bigcup_{n=0}^{\infty} E^{(2n}$, where $E \subseteq E^{**} \subseteq E^{(4} \subset \ldots \subseteq E^{(2n} \subseteq \ldots$.
In particular, for every $n\in \N$, we can define $\hat{P}:\bigcup_{m\geq n}^{\infty} E^{(2m} \longrightarrow E^{(2n}$
by the formula $\hat{P}(x)=x|_{E^{(2n-1}}$ for every $x\in \bigcup_{m\geq n}^{\infty} E^{(2m}$.
It is immediate that $\| \hat{P}\|=1$ and, since $E^{(\infty}$ can be seen as the completion of the normed space $\bigcup_{m\geq n}^{\infty} E^{(2m}$, $\hat{P}$ can be uniquely extended by completion to an operator $\acute{P}\colon E^{(\infty} \longrightarrow E^{(2n}$. Now, if we denote by $I\colon E^{(2n} \longrightarrow E^{(\infty}$ the canonical inclusion, then $P:=I\circ \acute{P} \colon E^{(\infty}\longrightarrow E^{(\infty}$
satisfies $P^2=P$, $P(E^{(\infty})=E^{(2n}$ and $\|P\|=1$. Thus, each $E^{(2n}$ is $1$-complemented in $E^{(\infty}$ for every $n \in \N$. It follows from \cite[Corollary 2.7]{ART18} that $FBL[E^{(2n}]$ is lattice isometric to a sublattice of $FBL[E^{(\infty}]$. Moreover, since $FBL[E]$ is lattice isometric to a sublattice of $FBL[E^{**}]$ by Lemma \ref{lemma:normanbidu}, we conclude that $FBL[E^{(2n}]$ is lattice isometric to a sublattice of $FBL[E^{(\infty}]$ for every $n \geq 0$.
Now we are ready to prove Theorem \ref{theo:cuninxInfiinfi}.

\begin{proof}[Proof of Theorem \ref{theo:cuninxInfiinfi}]  Suppose that $C(E^{(\infty})$ is infinite-dimensional. Let $k \in \N$ be fixed. Let $f_1, \ldots, f_k \in S_{FBL[E]}$ and $\eps > 0$ be given. We can see $f_i$ in $S_{FBL[E^{(\infty}]}$. Since $C(E^{(\infty})$ is infinite-dimensional, by Theorem \ref{theo:cunialgespa}, there exists $x \in S_{E^{(\infty}}$ such that $\|f_i + \delta_{x}\|_{FBL[E^{(\infty}]} > 2 - \eps$ for every $i=1,\ldots, k$. Since $\bigcup_{n=0}^{\infty} E^{(2n}$ is dense in $E^{(\infty}$, we may assume that $x \in S_{E^{(2n}}$ for some $n\in \N$. Let us denote $x$ by $x^{(2n}$ to indicate this element is in $E^{(2n}$. Therefore,
	\begin{equation*}
	\|f_i + \delta_{x^{(2n}}\|_{FBL[E^{(2n}]} > 2 - \eps
	\end{equation*}
	for every $i=1,\ldots, N$. By Lemma \ref{lemma:octabaja}, there exists $x^{(2n-2} \in B_{E^{(2n-2}}$ such that
\begin{equation*} 	
	\|f_i + \delta_{x^{(2n-2}}\|_{FBL[E^{(2n-2}]} > 2 - \eps
\end{equation*} 
holds for every $i=1,\ldots, k$. In order to finish the proof, we apply Lemma \ref{lemma:octabaja} finitely many times.
\end{proof}

\section{Dual Almost Square Banach spaces}\label{sect:ASQ}

In this section, we will be dealing with almost square Banach spaces. Recall that a Banach space $E$ is \textit{almost square} (ASQ, for short) if, for every $\{x_1,\ldots, x_n\}\subseteq S_E$ and every $\eps>0$, there exists a sequence $\left(y_k\right)_{k \in \N}$ in $S_E$ such that $\Vert x_i\pm y_k\Vert\rightarrow 1$ holds for every $i=1,\ldots,n$. Notice that the sequence $\left(y_k\right)_{k \in \N}$ can be chosen to be weakly null (actually equivalent to the $c_0$ basis (see \cite{all})). Almost squareness was introduced in \cite{all}, where an intensive study of it was done. It was proved, among other things, that ASQ spaces satisfy the SD2P and that they contain an isomorphic copy of $c_0$. In fact, later, in \cite{blrasq}, it was proved that a Banach space $E$ admits an equivalent ASQ renorming if and only if $E$ contains an isomorphic copy of $c_0$.

Under certain conditions related to the ASQ on a Banach space $E$, the projective norm on $E^* \pten F^*$ is octahedral (see  \cite[Corollary 2.9]{llr1}). The idea behind it can be translated in terms of the good behavior of the $c_0$-orthogonality in the injective tensor product, which is what happens with the free Banach lattice as we can see in the proof of the following result.

\begin{thm}\label{theo:dualASQ}
Let $E$ be a Banach space. If the dual $E^*$ is ASQ, then the norm of $FBL[E]$ is octahedral.
\end{thm}

\begin{proof} In what follows, we will prove that the dual of $FBL[E]$ satisfies the $w^*$-SD2P (see characterization \cite[Theorem 2.1]{blrjfa}). Consider a convex combination of $w^*$-slices 
\begin{equation*} 	
C=\sum_{i=1}^n \lambda_i S_i=\sum_{i=1}^n \lambda_i S(B_{FBL[E]},f_i,\alpha),
\end{equation*} 
where $S_i:=S(B_{FBL[E]},f_i,\alpha) = \{\phi\in B_{FBL[E]^*}:\phi(f)>1-\alpha\}$ with $\alpha>0$ and $f\in S_{FBL[E]}$. Let $\varepsilon>0$ be arbitrary. We need to find two elements $c_1,c_2\in C$ such that $\Vert c_1-c_2\Vert>2-\varepsilon$. To this end, fix $i\in\{1,\ldots, n\}$ and consider the set
\begin{equation*} 
A:=\left\{\sum_{i=1}^n \gamma_i \delta_{x_i^*}: x_i^* \in E^*, \gamma_i\in\{-1,1\}, \sup\limits_{x\in B_E}\sum_{i=1}^n \vert x_i^*(x)\vert\leq 1\right\}.
\end{equation*} 
Notice that, by the discussion around equation \eqref{equa:norming}, it follows that $A$ is norming for $FBL[E]$ and $\overline{A}^{w^*}=\overline{co}^{w^*}(A)=B_{FBL[E]^*}$. Since, for every $i\in\{1,\ldots, n\}$, $B_{FBL[E]^*}\setminus S_i$ is $w^*$-closed and convex, it follows that $A\cap S_i\neq \emptyset$. This means that there exists, for every $i\in\{1,\ldots, n\}$, an element $\sum_{j=1}^{n_i} \gamma_{ij} \delta_{x_{ij}^*}\in S_i\cap A$. Notice that the condition $\sup_{x\in B_E}\sum_{j=1}^{n_i} \vert x_{ij}
^*(x)\vert\leq 1$ is equivalent to the fact that $\Vert\sum_{j=1}^{n_i} \eta_j x_{ij}^*\Vert\leq 1$ holds for every choice of signs $\eta_1,\ldots, \eta_{n_i}$. Since $E^*$ is ASQ, by assumption, we can find a weakly null sequence $\left(y_k^*\right)_{k \in \N}$ such that
\begin{equation*} 
\lim\limits_{k}\left \Vert \sum_{j=1}^{n_i} \eta_j x_{ij}^*\pm y_k^*\right \Vert= 1
\end{equation*} 
for every $i\in\{1,\ldots, n\}$ and every choice of signs $\eta_j$. By Lemma \ref{LemNormInequalities}, this implies that 
\begin{equation*} 
	\lim\limits_{k}\left \Vert \sum_{j=1}^{n_i} \gamma_{ij} \delta_{x_{ij}^*}\pm \delta_{y_k^*}\right \Vert= 1
\end{equation*} 
for every $i\in\{1,\ldots, n\}$.
Moreover, notice that $y_k^*\stackrel{w}{\longrightarrow} 0$ implies that $y_k^*\stackrel{w^{*}}{\longrightarrow}0$. Therefore, $\delta_{y_k^*}\stackrel{w^{*}}{\longrightarrow} 0$. Thus, if we define
\begin{equation*} 
u_k^i:=\frac{\sum_{j=1}^{n_i} \gamma_{ij}\delta_{x_{ij}^*}+\delta_{y_k^*}}{\Vert \sum_{j=1}^{n_i} \gamma_{ij}\delta_{x_{ij}^*}+\delta_{y_k^*}\Vert},
\end{equation*} 
and
\begin{equation*} 
v_k^i:=\frac{\sum_{j=1}^{n_i} \gamma_{ij}\delta_{x_{ij}^*}-\delta_{y_k^*}}{\Vert \sum_{j=1}^{n_i} \gamma_{ij}\delta_{x_{ij}^*}-\delta_{y_k^*}\Vert},
\end{equation*} 
then $u_k^i,v_k^i$ are sequences in $S_{FBL[E]^*}$, which converge to $\sum_{j=1}^{n_i} \gamma_{ij}\delta_{x_{ij}^*}$. So, we can find $m\in\mathbb N$ such that $u_k^i,v_k^i\in S_i$ holds for every $k\geq m$ and every $i\in\{1,\ldots, n\}$. So, 
\begin{equation*} 
\diam(C)\geq \left\| \sum_{i=1}^n \lambda_i (u_k^i-v_k^i) \right\| 
\end{equation*} 
for every $k\geq m$. Let us estimate $\Vert \sum_{i=1}^n \lambda_i (u_k^i-v_k^i)\Vert$. To do so, in order to save notation, let us set 
\begin{equation*} 
\alpha_{ki}^{\pm}=\left\| \sum_{j=1}^{n_i} \gamma_{ij}\delta_{x_{ij}^*}\pm\delta_{y_k^*} \right\|.
\end{equation*} 
Now, given $k\geq m$, we get that 
\begin{equation*} 
\sum_{i=1}^n \lambda_i (u_k^i-v_k^i)=\delta_{y_k^*} \sum_{i=1}^n \lambda_i\left( \frac{1}{\alpha_{ki}^+}+\frac{1}{\alpha_{ki}^-} \right)+\sum_{i=1}^n \lambda_i\sum_{j=1}^{n_i}\gamma_{ij} \delta_{x_{ij}^*} \left( \frac{1}{\alpha_{ki}^+}-\frac{1}{\alpha_{ki}^-} \right).
\end{equation*} 
Notice that $\alpha_{ki}^\pm\rightarrow 1$ for every $i\in\{1,\ldots, n\}$. Hence,
\begin{eqnarray*}
	\left \Vert \sum_{i=1}^n \lambda_i(u_k^i-v_k^i)\right \Vert &\geq& \sum_{i=1}^n \lambda_i \left (\frac{1}{\alpha_{ki}^+}+\frac{1}{\alpha_{ki}^-}\right )-\left \Vert \sum_{i=1}^n \left \vert\frac{1}{\alpha_{ki}^+}-\frac{1}{\alpha_{ki}^-}\right \vert \lambda_i \sum_{j=1}^{n_i} \gamma_{ij} \delta_{x_{ij}^*}\right \Vert  \\
	&\geq&  \sum_{i=1}^n \lambda_i \left (\frac{1}{\alpha_{ki}^+}+\frac{1}{\alpha_{ki}^-}\right )- \sum_{i=1}^n \left \vert\frac{1}{\alpha_{ki}^+}-\frac{1}{\alpha_{ki}^-}\right \vert \lambda_i	\longrightarrow 2.
\end{eqnarray*} 
\end{proof}

\begin{rem}\label{remark:johann}
In the proof of Theorem \ref{theo:dualASQ}, we actually obtain that $FBL[E]^*$ has the $w^*$-symmetric strong diameter two property, a stronger property than the $w^*$-SD2P. Recall that a Banach space $E$ has the \textit{symmetric strong diameter two property} (\emph{SSD2P} in short) if for every $n\in\mathbb N$, every slices $S_1,\ldots, S_n$ of $B_E$ and every $\varepsilon>0$, there are $x_i\in S_i$ for every $i\in\{1,\ldots, n\}$ and there exists $\varphi\in B_E$ with $\Vert \varphi\Vert>1-\varepsilon$ such that $x_i\pm \varphi\in S_i$ for every $i\in\{1,\ldots, n\}$ (see \cite{hlln} for examples and background and \cite{RZ} for applications of this property). If $E$ is a dual Banach space, the \textit{$w^*$-symmetric strong diameter two property ($w^*$-SSD2P)} is defined in a similar way just replacing slices with $w^*$-slices. The authors of the present paper thank Johann Langemets for pointing out this remark.
\end{rem}

Let us observe that it was posed in \cite[Question 6.6]{all} as an open question whether the dual of a given Banach space $E$ can be ASQ. A positive answer has been recently given in \cite{aht}. Even though, to show that Theorem \ref{theo:dualASQ} applies to a large class of Banach spaces, notice that in \cite[Theorem 3.8]{aht} it was proved that if a Banach space $E$ satisfies that $E^*$ contains an isomorphic copy of $c_0$, then there exists an equivalent renorming of $E$, say $F$, so that $F^*$ is ASQ.

\section{Remarks and open questions} \label{sect:opeque}

In this section, we will discuss some open questions derived from our work. Recall that a Banach space $E$  is said to have the {\it Daugavet property} if every rank-one operator $T:E\longrightarrow E$ satisfies the equality
\begin{equation*}
\Vert T+I\Vert=1+\Vert T\Vert,
\end{equation*}
where $I$ denotes the identity operator on $E$. Some examples of Banach spaces enjoying the Daugavet property are $\mathcal C(K)$ for a compact Hausdorff and perfect topological space $K$, $L_1(\mu)$, and $L_\infty(\mu)$ for a non-atomic measure $\mu$ or the space of Lipschitz functions $Lip(M)$ over a metrically convex space $M$ (see \cite{ikw,kssw,sh,werner} and the references therein for details). It is known that if $E$ has the Daugavet property, then the norms of $E$ and $E^*$ are octahedral \cite[Lemmas 2.8 and 2.12]{kssw}.

In the recent paper \cite{rtv20}, it has been proved that if $E$ and $F$ are $L_1$-preduals with the Daugavet property, then $E\widehat{\otimes}_\pi F$ has the Daugavet property, based on the celebrated characterization of $L_1$-preduals given in \cite{linds} in terms of the existence of extensions of compact operators. Since the results of this paper are strongly motivated by the known results about octahedrality in tensor product spaces, we wonder the following.

\begin{question}
Let $E$ be an $L_1$-predual with the Daugavet property. Is the norm of $FBL[E]$ octahedral?
\end{question}

Even though we do not know the answer, we can at least prove the following result, that say that we can ensure octahedrality condition when dealing with positive elements in the following sense.

\begin{prop}
Let $E$ be an $L_1$-predual with octahedral norm and denote by $S_{FBL[E]}^+$ the positive elements of $S_{FBL[E]}$. Then, for every $f_1,\ldots, f_n\in S_{FBL[E]}^+$ and every $\varepsilon>0$, there exists $x\in S_E$ such that
$$\Vert f_i+\vert \delta_x\vert\Vert>2-\varepsilon.$$
\end{prop}

\begin{proof}
Let $i\in\{1,\ldots, n\}$. Consider $x_{ij}^*\in E^*$ and $\xi_{ij}\in \{-1,1\}$ such that 
\begin{equation*} 
\sum_{j=1}^{n_i} \xi_{ij} f_i(x_{ij}^*)>1-\varepsilon \ \ \ \mbox{with} \ \ \ \sup_{x\in B_E}\sum_{j=1}^{n_i} |x_{ij}^*(x)|=1. 
\end{equation*} 
Since $f_i$ is positive we deduce that 
\begin{equation*}
\sum_{j=1}^{n_i}f_i(x_{ij}^*)>1-\varepsilon.
\end{equation*} 
By a density argument, we can assume, without loss of generality, that $f_i$ depends on finitely many coordinates (see Remark \ref{RemDensidadFuncionesQueDependenDeUnaCantidadFinitaDeCoordenadas}). This means that there exists a finite-dimensional subspace $F$ of $E$ such that if $x^*|_{F}=y^*|_{F}$, then $f_i(x^*)=f_i(y^*)$ for every $i\in\{1,\ldots, n\}$. Now, let $x_i\in B_E$ be such that 
\begin{equation*} 
\sum_{j=1}^{n_i} \vert x_{ij}^*(x_i)\vert>1-\varepsilon.
\end{equation*}
Since the norm of $E$ is octahedral, we can find, for $\delta>0$, an element $x\in S_E$ such that
$$\Vert e+\lambda x\Vert\geq (1-\delta)(\Vert e\Vert+\vert \lambda\vert)$$
holds for every $e\in F$ and every $\lambda\in\mathbb R$. Define, for every $i\in\{1,\ldots, n\}$, an operator $T_i:F_x \longrightarrow E$ by the equation
$$T_i(e+\lambda x)=e+\lambda x_i,$$
where $e\in F$ and $F_x$ denotes the subspace of $E$ generated by $F$ and $x$.
We claim that $\Vert T_i\Vert\leq \frac{1}{1-\delta}$. Indeed,
$$\Vert T_i(e+\lambda x)\Vert=\Vert e+\lambda x_i\Vert\leq \Vert e\Vert+\vert \lambda\vert\leq \frac{1}{1-\delta}\Vert e+\lambda x\Vert.$$
Since $T_i\colon F_x \longrightarrow E$ is a compact operator and $E$ is an $L_1$-predual, there are compact extensions (denoted in the same way) $T_i:E\longrightarrow E$ such that $\Vert T_i\Vert\approx \frac{1}{1-\delta}$ (see \cite[Theorem 6.1]{linds}). Choosing $\delta$ small enough we can assume that $\Vert T_i\Vert\leq 1+\varepsilon$. Define $y_{ij}^*:=x_{ij}^*\circ T_i$. Let us prove that $\sup_{v\in B_E} \sum_{j=1}^{n_i}\vert y_{ij}^*(v)\vert\leq 1+\varepsilon$. To this end pick $v\in B_E$ and notice that, since $\Vert T_i\Vert\leq 1+\varepsilon$, then $\frac{T_i(v)}{1+\varepsilon}\in B_E$. Hence
$$1\geq \sum_{j=1}^{n_i}\left\vert x_{ij}^*\left(\frac{T_i(v)}{1+\varepsilon}\right)\right\vert=\frac{\sum_{j=1}^{n_i} \vert y_{ij}^*(v)\vert}{1+\varepsilon}.$$
Thus, by the definition of the norm of $FBL[E]$, this implies that $\sum_{j=1}^{n_i} \delta_{y_{ij}^*}\in (1+\varepsilon)B_{FBL[E]^*}$. Define $f:=\vert \delta_x\vert\in S_{FBL[E]}$. Since $y_{ij}^*=x_{ij}^*$ on $F$, then $f_i(x_{ij}^*)=f_i(y_{ij}^*)$. Therefore,
\begin{eqnarray*} 
(1+\varepsilon)\Vert f_i+f\Vert \geq \sum_{j=1}^{n_i}\delta_{y_{ij}^*}(f_i+f) &=& \sum_{j=1}^{n_i}(f_i(x_{ij}^*)+\vert y_{ij}^*(x)\vert)\\
&>& 1-\varepsilon+\sum_{j=1}^{n_i}\vert x_{ij}^*(x_i)\vert \\
&>& 2-2\varepsilon.
\end{eqnarray*} 
\end{proof}

Concerning the Daugavet property, it is natural to pose the following question.

\begin{question}
Does $FBL[E]$ have the Daugavet property for any Banach space $E$?
\end{question}

Let us explain why we cannot deduce the Daugavet property from our results. Let $E$ be a Banach space, take $f\in FBL[E]$, and $\varepsilon>0$. Consider the set
$$A(f,\varepsilon)=\{g\in B_{FBL[E]}: \Vert f+g\Vert>2-\varepsilon\}.$$
Then, $FBL[E]$ has the Daugavet property if and only if $\overline{co}^{w^*}(A(f,\varepsilon))=B_{FBL[E]}$ for every $f \in B_{FBL[E]}$ and $\eps > 0$ (see, e.g. \cite[Corollary 2.3]{werner}). However, in the proof of Theorem \ref{theo:cuninxInfiinfi}, the elements of $A(f,\varepsilon)$ that we exhibit are of the form $\delta_x$ for $x\in B_X$, which clearly are not enough to generate the whole unit ball by taking closed convex hull. On the other hand, let us point out that, to the best of our knowledge, it is an open problem (see \cite[Section 6, (3)]{werner}) whether the Daugavet property is stable under projective tensor products of two Banach spaces satisfying both the Daugavet property.

\bigskip 

Finally, we wonder if there are free Banach lattices which do not have octahedral norms. Since every function in $FBL[E]$ is positively homogeneous, it follows that $FBL[E]$ has dimension 2 whenever $E$ has dimension $1$. Thus, a natural example of a free Banach lattice without octahedral norm is $FBL[E]$ for $E$ one-dimensional. Nevertheless, if $\dim(E)\geq 2$, then $FBL[E]$ is infinite-dimensional. For this reason, we wonder the following.

\begin{question} \label{finite} 
Let $E$ be a finite-dimensional Banach space of dimension $\geq2$. Is the norm of $FBL[E]$ octahedral? 
\end{question}

Even though we do not know the answer for Question \ref{finite}, let us point out that convex combination of slices of $B_{FBL[E]^*}$ cannot have arbitrarily small diameter.

\begin{prop}\label{prop:bigccslice}
	Let $E$ be a finite-dimensional Banach space with dimension $n\geq 2$. Let $\beta$ be the Banach-Mazur distance between $\ell_1(n)$ and $E$ and set $\alpha:=\frac{2}{ne^\beta}$. Then, every convex combination of $w^*$-slices of the dual ball $B_{FBL[E]^*}$ has diameter greater or equal than $\alpha$.
\end{prop}

\bigskip

Following the notation of \cite[Definition 1.10]{dgz}, recall that the norm of a Banach space $E$ is said to be \textit{$\varepsilon$-rough} if, for every $x\in E$, it follows that
$$\limsup\limits_{\Vert h\Vert\rightarrow 0}\frac{\Vert x+h\Vert+\Vert x-h\Vert-2\Vert x\Vert}{\Vert h\Vert}\geq \varepsilon.$$
We say that the norm of $E$ is \textit{rough} if it is $\varepsilon$-rough for some $\varepsilon>0$. Notice that roughness condition means ``uniformly nowhere Fr\'echet differentiable''. 
\cite[Proposition 1.11]{dgz} states that Banach spaces which do not have $w^*$-slices of arbitrarily small diameter are rough. 
Thus, the following corollary follows from Proposition \ref{prop:bigccslice}.

\begin{cor}\label{Cororudez}
		Let $E$ be a finite-dimensional Banach space with dimension $n\geq 2$. Then, the norm of $FBL[E]$ is rough. In particular, it is nowhere Fr\'echet differentiable.
\end{cor}

\bigskip

The case $E=\ell_1(\Gamma)$ has a special interest, since in this case the free Banach lattice generated by $E$ coincides with the free Banach lattice generated by the set $\Gamma$  (see \cite[Corollary 2.8]{ART18}). For this case the following holds.
\begin{cor}
	Let $\Gamma$ be a set with more than one point. Then, the norm of $FBL[\ell_1(\Gamma)]$ is rough. In particular, it is nowhere Fr\'echet differentiable.
\end{cor}
\begin{proof}
When $\Gamma$ is finite, the result follows from Corollary \ref{Cororudez}. If $\Gamma$ is infinite then $FBL[\ell_1(\Gamma)]$ has an octahedral norm by Theorem \ref{theo:cuninxInfiinfi} and Example \ref{exam:aplicunin}. Consequently, the norm of $FBL[\ell_1(\Gamma)]$ is $2$-rough (see \cite[p.~78]{dgz}).
\end{proof}

\bigskip

In order to prove Proposition \ref{prop:bigccslice}, we need some technical lemmas. In \cite[Proposition 5.3]{dPW15} de Pagter and Wickstead proved that free Banach lattices generated by finite sets are lattice isomorphic to $\mathcal{C}(K)$-spaces. We reproduce part of the idea of the proof, in terms of $FBL[\ell_1(n)]$, in the paragraph below for the sake of completeness and to deal with the constants and the explicit isomorphism appearing in the proof.

Let $n\in\N$ be fixed. Notice that $S_{\ell_\infty(n)}$ is a compact space with the $w^*$-topology.
The operator $R\colon FBL[\ell_1(n)] \longrightarrow C(S_{\ell_\infty(n)})$ given by $Rf=f|_{S_{\ell_\infty(n)}}$ is a surjective Banach lattice isomorphism. De Pagter and Wickstead proved that 
$$  \frac{1}{n} \| f \| \leq  \|R(f) \|_\infty \leq || f || $$
for every $f \in FBL[\ell_1(n)]$. The fact that $R$ is surjective follows from the lattice version of the Stone-Weierstrass Theorem, since $R(FBL[\ell_1(n)])$ is a closed (linear) sublattice of $C(S_{\ell_\infty(n)})$ that separates points, and contains the constant function $1=R( |\delta_{e_1}|\vee|\delta_{e_2}|\vee \ldots \vee |\delta_{e_n}|)$. Indeed, since $\||\delta_{e_1}|\vee|\delta_{e_2}|\vee \ldots \vee |\delta_{e_n}|\|=n$ (see the comments below \cite[Proposition 5.3]{dPW15}),  it follows from the monotony of the norm and the fact that $f \leq 1 $ for every $f \in  B_{C(S_{\ell_\infty(n)})}$ that
\begin{equation*}
\|R\|=1 \mbox{ and } \|R^{-1}\|=n.
\end{equation*}

\noindent
With the previous notation in mind, we get the following lemmas.

\begin{lem}
	\label{LemNorma}
	Let $n\in \N$. Then, for every $\gamma_1,\ldots, \gamma_m \in \R$, we have $$ \left \| \sum_{i=1}^m \gamma_i \delta_{x_i^*}\right \| \geq \frac{1}{n} \sum_{i=1}^{m} |\gamma_i|,$$
	whenever $x_1^*,\ldots,x_m^* \in S_{\ell_\infty(n)}$ are all different from each other.  
\end{lem}
\begin{proof}
For each $x_i^* \in S_{\ell_\infty(n)}$, set $\tilde{\delta}_{x_i^*} \in C(S_{\ell_\infty(n)})$ the corresponding evaluation functional. Notice that 
\begin{equation*} 
\left  \| \sum_{i=1}^m \gamma_i \tilde{\delta}_{x_i^*}\right  \| =  \sum_{i=1}^m |\gamma_i|
\end{equation*} 
Moreover,
\begin{equation*} 
R^*(\tilde{\delta}_{x_i^*})(f)=x_i^*(Rf)=f(x_i^*)=\delta_{x_i^*}(f) \mbox{ for every } f \in FBL[\ell_1(n)],
\end{equation*} 
i.e., $R^*(\tilde{\delta}_{x_i^*})=\delta_{x_i^*}$ for every $i \leq m$. But then, bearing in mind that $\|(R^*)^{-1}\|=n$, we conclude that 
\begin{equation*} 
\left\| \sum_{i=1}^m \gamma_i \delta_{x_i^*} \right\| = \left\| R^*\left(\sum_{i=1}^m \gamma_i  \tilde{\delta}_{x_i^*}\right)\right \| \geq \frac{1}{n}\left \| \sum_{i=1}^m \gamma_i  \tilde{\delta}_{x_i^*} \right \| = \frac{1}{n} \sum_{i=1}^m |\gamma_i|.
\end{equation*} 
\end{proof}

\begin{lem}
	\label{ThmNormaEfinitedimensional}
	Let $n\in \N$ and $E$ be a Banach space of dimension $n$. Let $\beta$ be the Banach-Mazur distance between $E$ and $\ell_1(n)$. Then,  for every $\gamma_1,\ldots, \gamma_m \in \R$, we have $$ \left \| \sum_{i=1}^m \gamma_i \delta_{x_i^*}\right \| \geq \frac{1}{n e^\beta} \sum_{i=1}^{m} |\gamma_i|\|x_i^*\|,$$
	whenever $x_1^*,\ldots,x_m^* \in E^*$ are pairwise independent.  
\end{lem}

\begin{proof}
Let $T\colon \ell_1(n) \longrightarrow E$ be an isomorphism. We have that $\frac{\|x\|}{\|T^{-1}\|} \leq \|Tx\| \leq \|T\| \|x\|$ for every $x\in \ell_1(n)$.
By the universal property of free Banach lattices, $T$ extends uniquely to a lattice homomorphism $\hat{T}\colon FBL[\ell_1(n)] \longrightarrow FBL[E]$ such that $\hat{T}(\delta_x)=\delta_{Tx}$ for every $x\in \ell_1(n)$ and $\|\hat{T}\|=\|T\|$.
Moreover, $\hat{T}^*(\delta_{x^*})=\delta_{T^*x^*}$ for every $x^* \in E^*$. 
Thus, if $x_1^*,\ldots,x_m^* \in E^*$ are pairwise independent elements, then
\begin{equation*} 
\left \| \sum_{i=1}^m \gamma_i \delta_{T^*x_i^*} \right\| =\left \| \hat{T}^*\left(\sum_{i=1}^m \gamma_i \delta_{x_i^*}\right) \right\| \leq \|T\| \left\| \sum_{i=1}^m \gamma_i \delta_{x_i^*} \right\|.
\end{equation*}
Now, by Lemma \ref{LemNorma}, we have
\begin{eqnarray*} 
\left \| \sum_{i=1}^m \gamma_i \delta_{T^*x_i^*} \right\| &=& \left \| \sum_{i=1}^m \gamma_i \|T^*x_i^*\|\delta_{\frac{T^*x_i^*}{\|T^*x_i^*\|}} \right\| \\
&\geq& \frac{1}{n} \sum_{i=1}^m |\gamma_i| \|T^*x_i^*\| \\
&\geq&  \frac{1}{n\|T^{-1}\|} \sum_{i=1}^m |\gamma_i| \|x_i^*\| \\
&=& \frac{1}{n\|T^{-1}\|} \sum_{i=1}^m |\gamma_i| \|\delta_{x_i^*}\|.
\end{eqnarray*} 
Thus, 
\begin{equation*} 
\left\| \sum_{i=1}^m \gamma_i \delta_{x_i^*} \right\| \geq  \frac{1}{\|T\|} \left \| \sum_{i=1}^m \gamma_i \delta_{T^*x_i^*} \right\| \geq\frac{1}{n\|T^{-1}\|\|T\|} \sum_{i=1}^m |\gamma_i| \|\delta_{x_i^*}\|.
\end{equation*} 
Since the last inequality holds for any isomorphism $T\colon \ell_1(n) \longrightarrow E$, the conclusion follows from the definition of the Banach-Mazur distance.
\end{proof}	

\begin{proof}[Proof of Proposition \ref{prop:bigccslice}]

Set $C :=\sum_{i=1}^k \lambda_i S_i$ to be any convex combination of $w^*$-slices of $B_{FBL[E]^*}$. 
Recall that the set \begin{equation*} 
	A:= \left\{\sum_{i=1}^n \xi_i\delta_{x_i^*} \in FBL[E]^*: x_i^*\in E^*, \xi_i\in \{-1,1\}, \sup\limits_{x\in B_E} \sum_{i=1}^n \vert x_i^*(x)\vert\leq 1 \right\}.
\end{equation*} 
is norming, so $B_{FBL[E]^*}=\overline{co(A)}^{w^*}$. Since $B_{FBL[E]^*}\setminus S_i$ is convex and $w^*$-closed, this implies that, for every $i\in\{1,\ldots, k\}$, there exists an element $\sum_{j=1}^{n_i}\gamma_{ij} \delta_{x_{ij}^*}\in A\cap S_i$.
Now, since $n>1$, we can find pairwise independent functions $y_{ij}^*, z_{ij}^* \in B_{E^*}$ for every $i\in\{1,\ldots, k\}$ and $j\in\{1,\ldots, n_i\}$ such that each $y_{ij}^*$ and each $z_{ij}^*$ are arbitrarily close to $x_{ij}^*$. In particular, since the functions in $FBL[E]$ are norm continuous (because $E$ is finite-dimensional), we can assume that $\sum_{j=1}^{n_i} \gamma_{ij}' \delta_{y_{ij}^*} $ and $\sum_{j=1}^{n_i} \gamma_{ij}'' \delta_{z_{ij}^*} $ belong to $S_i$ and have norm one for every $i\in\{1,\ldots, k\}$. Now, it follows from Lemma \ref{ThmNormaEfinitedimensional} that
\begin{equation*} 
\left \| \sum_{i=1}^k\lambda_i \sum_{j=1}^{n_i} \gamma_{ij}' \delta_{y_{ij}^*} - \sum_{i=1}^k \lambda_i \sum_{j=1}^{n_i} \gamma_{ij}'' \delta_{z_{ij}^*} \right \| \geq \frac{1}{n e^\beta} \sum_{i=1}^k\sum_{j=1}^{n_i} \lambda_i  (|\gamma_{ij}'| \|y_{ij}^*\|+|\gamma_{ij}''| \|z_{ij}^*\|).
\end{equation*} 
But 
\begin{equation*} 
1= \left\|\sum_{j=1}^{n_i} \gamma_{ij}' \delta_{y_{ij}^*} \right\| \leq \sum_{j=1}^{n_i} |\gamma_{ij}'| \|y_{ij}^*\| 
\end{equation*} 
and, analogously,  
\begin{equation*} 
1= \left\| \sum_{j=1}^{n_i} \gamma_{ij}'' \delta_{z_{ij}^*} \right\| \leq \sum_{j=1}^{n_i} |\gamma_{ij}''| \|z_{ij}^*\|.
\end{equation*}
Thus, we conclude that 
\begin{equation*} 
\left \| \sum_{i=1}^k\lambda_i \sum_{j=1}^{n_i} \gamma_{ij}' \delta_{y_{ij}^*} - \sum_{i=1}^k \lambda_i  \sum_{j=1}^{n_i} \gamma_{ij}'' \delta_{z_{ij}^*}  \right \| \geq \frac{1}{ne^\beta} \sum_{i=1}^k 2\lambda_i  =\frac{2}{ne^\beta}=\alpha
\end{equation*} 
and so $C$ has diameter at least $\alpha$, as desired.
\end{proof}

\section*{Acknowledgements} We thank Pedro Tradacete for the fruitful discussions held on the early phase of this work. We also thank Antonio Avil\'es and Jos\'e Rodr\'iguez for pointing out typos and for comments that have improved the exposition of the text. Finally, we thank Johann Langemets for pointing out Remark \ref{remark:johann}.

\setlength{\parskip}{4mm}

%
%

\end{document}